\documentclass[preprint,12pt]{elsarticle}

\usepackage{amssymb}
\usepackage{amsthm}
\usepackage{ amsfonts}
\usepackage{amsmath}

\newtheorem{thm}{Theorem}

\newtheorem{lemma}{Lemma}
\newdefinition{defi}{Definition}
\newdefinition{remark}{Remark}
\biboptions{sort,compress}

\begin{document}

\begin{frontmatter}

  \title{Comparison of probabilistic and deterministic point sets}

\author{Peter Grabner\fnref{thanks}\corref{cor}}
\ead{peter.grabner@tugraz.at}
\author{Tetiana Stepanyuk\fnref{thanks}}
\ead{t.stepaniuk@tugraz.at}
\address{Graz University of Technology, Institute of Analysis and Number
  Theory, Kopernikusgasse 24/II 8010, Graz, Austria}
\fntext[thanks]{The authors are supported by the Austrian Science Fund FWF
  project F5503 (part of the Special Research Program (SFB) 
``Quasi-Monte Carlo Methods: Theory and Applications'')}

\cortext[cor]{Corresponding author}

\begin{abstract}
  In this paper we make a comparison between certain probabilistic and deterministic point sets and show that some deterministic constructions (spherical $t$-designs) are better or as good as probabilistic ones. 
  
  We find asymptotic equalities for the  discrete Riesz $s$-energy of sequences of well separated $t$-designs on the unit sphere $\mathbb{S}^{d}\subset\mathbb{R}^{d+1}$, $d\geq2$.
  The case $d=2$ was studied in \cite{HesseLeopardiTheCoulombEnergy, Hesse:2009s-energy}. In \cite{Bondarenko-Radchenko-Viazovska2013:optimal_designs} it was established, that for $d\geq 2$, there exists a constant $c_{d}$, such that for every $N> c_{d}t^{d}$ there exists a well-separated spherical $t$-design on $\mathbb{S}^{d}$ with $N$ points. For this reason, in our paper we assume, that the sequence of well separated spherical $t$-designs is such that $t$ and $N$ are related by $N\asymp t^{d}$.



%
\end{abstract}


\end{frontmatter}

\section{Introduction}\label{intro}

Let $\mathbb{S}^{d}=\{\mathbf{x}\in\mathbb{R}^{d+1}: \ |\mathbf{x}|=1\}$, where $d\geq2$, be the unit
sphere in the Euclidean space $\mathbb{R}^{d+1}$, equipped with the Lebesgue measure 
$\sigma_{d}$ normalized by $\sigma_{d}(\mathbb{S}^{d})=1$.

 Let $K_{d}$ be the positive definite function (see \cite{Schoenberg:1942PositiveDefinite})
\begin{align}\label{kernel}
K_{d}(t):=\sum\limits_{n=0}^{\infty}a_{n}P_{n}^{(d)}(t), \ \ a_{n}\geq 0,
\end{align}
 where $P_{n}^{(d)}$ is the $n$-th generalized Legendre polynomial,
 normalized by ${P_{n}^{(d)}(1)=1}$ and orthogonal on the interval $[-1,1]$
 with respect to the weight function $(1-t^{2})^{d/2-1}$.
In this paper
 we investigate energy integrals with respect to a probabilistic model ("jittered sampling")  of the form
\begin{equation}\label{nonsingularInt}
\frac{1}{N^{2}}\int\limits_{A_{1}}...\int\limits_{A_{N}}\sum\limits_{i,j=1}^{N}K_{d}(\langle\textbf{x}_{i},\mathbf{x}_{j}\rangle)d\sigma_{1}^{*}(\mathbf{x}_{1})...d\sigma_{N}^{*}(\mathbf{x}_{N}), \ \ K_{d}\in \mathbb{C}_{[-1,1]}
\end{equation}
and
\begin{equation}\label{singularInt}
\frac{1}{N^{2}}\int\limits_{A_{1}}...\int\limits_{A_{N}}
\sum\limits_{i,j=1, \atop i\neq j}^{N}K_{d}(\langle\textbf{x}_{i},\mathbf{x}_{j}\rangle)d\sigma_{1}^{*}(\mathbf{x}_{1})...d\sigma_{N}^{*}(\mathbf{x}_{N}), \ \ K_{d}\in \mathbb{C}_{[-1,1)}.
\end{equation}

Here $\{A_{i}\}_{i=1}^{N}$ is an area regular partition of the sphere (see, e.g., \cite{RakhmanovSaffZhou:1994Minimal}), i.e.:
$\mathbb{S}^{d}=\bigcup\limits_{i=1}^{N}A_{i}$, $A_{i}\cap A_{j}=\emptyset, \ i\neq j$, $\sigma(A_{i})=\frac{1}{N}$
 and the point $ \mathbf{x}_{i}$ is chosen uniformly randomly in $A_{i}$ for $i=1,...,N$.

 We denote
\begin{align}\label{energy}
E(K_{d}, X_{N}):=\frac{1}{N^{2}}\sum\limits_{i,j=1}^{N}K_{d}(\langle\mathbf{x}_{i},\mathbf{x}_{j}\rangle).
\end{align}
 
\begin{defi}\label{def1}
  A spherical $t$-design is a finite subset $X_{N}\subset \mathbb{S}^{d}$ with
  a characterizing property that an equal weight integration rule with nodes
  from $X_{N}$ integrates all spherical polynomials $p$ of total degree at most  $t$ exactly;
  that is,
\begin{equation*}
  \frac{1}{N}\sum\limits_{\mathbf{x}\in X_{N}}p(\mathbf{x})=
  \int_{\mathbb{S}^{d}}p(\mathbf{x})d\sigma_{d}(\mathbf{x}), \quad
  \mathrm{deg}(p)\leq t.
\end{equation*}

Here $N$ is the cardinality of $X_{N}$ or the number of points of spherical design.
\end{defi}

\begin{defi}
  A sequence of $N$-point sets $(X_{N})_{N}$,
  $X_{N}=\big\{\mathbf{x}_{1},\ldots, \mathbf{x}_{N} \big\}$, is called
  well-separated if there exists a positive constant $c_{1}$ such that
\begin{equation}\label{wellSeparat}
  \min\limits_{i\neq j}|\mathbf{x}_{i}- \mathbf{x}_{j}|>
  \frac{c_{1}}{N^{\frac{1}{d}}}.
\end{equation}
\end{defi}

The concept of spherical $t$-design was introduced  by Delsarte, Goethals and
Seidel in the groundbreaking paper \cite{Delsarte-Goethals-Seidel1977:spherical_designs}, where they also proved the lower bound $N\geq C_{d}t^{d}$.
The relation between $N$ and $t$ in spherical designs plays important role. Korevaar and Meyers \cite{KorevaarMeyers:SphericalFaraday1993} conjectured that there always exist spherical $t$-design with $N\asymp t^{d}$ points.

 We write
 $a_{n}\asymp b_{n}$ to mean that there exist positive constants $C_{1}$ and
 $C_{2}$ independent of $n$ such that $C_{1}a_{n}\leq b_{n}\leq C_{2}a_{n}$ for
 all $n$.
 
 Also many authors  have predicted the existence of well-separated spherical $t$-designs in 
 $\mathbb{S}^{d}$ of asymptotically minimal cardinality $\mathcal{O}(t^{d})$ as $t\rightarrow\infty$ (see, e.g., \cite{AnChenSloanWomersley_Well2010}, \cite{HesseLeopardiTheCoulombEnergy}).

In 2013  Bondarenko, Radchenko and Viazovska
\cite{Bondarenko-Radchenko-Viazovska2013:optimal_designs} proved, that indeed for ${d\geq 2}$, there exists a constant $c_{d}$, which depends only of $d$, such that for every $N\geq c_{d}t^{d}$ there exists a spherical $t$-design on $\mathbb{S}^{d}$ with $N$ points. Two years later in 
\cite{Bondarenko-Radchenko-Viazovska2015:Well_separated} they showed, that for each ${d\geq 2}$, $t\in \mathbb{N}$, $N>c_{d}t^{d}$, there exist positive constants  $c_{d}$ and $\lambda_{d}$, depending only on $d$, such that for every $N\geq c_{d}t^{d}$, there exists a spherical $t$-design on $\mathbb{S}^{d}$ ,
consisting of $N$ points $\{\mathbf{x}_{i}\}_{i=1}^{N}$ with $|\mathbf{x}_{i}-\mathbf{x}_{j}|\geq \lambda_{d}N^{-\frac{1}{d}}$ for $i\neq j$, where $c_{d}$ and $\lambda_{d}$ are positive constants, depending only on $d$. 

  Taking this into account we always assume that
 \begin{equation}\label{NandT}
 N=N(t)\asymp t^{d}.
 \end{equation}

For given $s>0$ the discrete Riesz $s$-energy of a set of $N$ points  $X_{N}$ on 
$\mathbb{S}^{d}$ is defined as
\begin{equation}\label{RieszbDef}
  E_{d}^{(s)}(X_{N}):=\frac{1}{2}
  {\mathop{\sum}\limits_{i,j=1,\atop i\neq j}^{N}}|\mathbf{x}_{i}-\mathbf{x}_{j}|^{-s},
\end{equation}
where $|\mathbf{x}|$ denotes the Euclidian norm in $\mathbb{R}^{d+1}$ of the vector $\mathbf{x}$. In the case $s=d-1$ the energy (\ref{RieszbDef}) is called as Coulomb energy.

In this paper we investigate and compare the asymptotic behaviour of the $s$-energy, for $0<s<d$ for sequences of well-separated $t$-designs and also for jittered sampling. 

Hesse and Leopardi
\cite{HesseLeopardiTheCoulombEnergy} showed,  that if spherical $t$-designs with 
${N=\mathcal{O}(t^{2})}$ exist, then they have asymptotically minimal Coulomb energy $E_{2}(X_{N})$. Namely, it was  proved, that the Coulomb energy of each $N$-point spherical $t$-design $X_{N}$ with the following properties: there exist positive constants $\mu$ and separation constant $\lambda$, such that $N\leq\mu(t+1)^{2}$, and the minimum spherical distance between point of $X_{N}$ is bounded from below by $\frac{\lambda}{\sqrt{N}}$, is bounded from above by
\begin{equation}\label{HesseandLeopardi}
  E_{2}^{(1)}(X_{N})\leq \frac{1}{2}N^{2}+C_{\lambda,\mu}N^{\frac{3}{2}}.
\end{equation}
Here and further by $C_{a}$ ($C_{a,b}$) we will denote constants, which may depend only on $a$ ($a$ and $b$), but not on $N$.

In \cite{Hesse:2009s-energy} the result (\ref{HesseandLeopardi}) was extended for all $0<s<2$. In particular, under the assumption that $N\leq \kappa t^{2}$, it was shown that for $0<s<2$, there exists a positive constant $c_{s}$ such that for every  well separated sequence $N$
-point spherical $t$-designs the following estimate holds
\begin{equation}\label{Hesse}
  E_{2}^{(s)}(X_{N})\leq \frac{2^{-s}}{2-s}N^{2}+C_{s,\kappa}N^{1+\frac{s}{2}}.
\end{equation}

Also it should be noticed, that in \cite{BoyvalenkovDragnevHardinSaffStoyanova} some general upper and lower bounds for the energy of spherical designs were found.

The separation constraint was important for the  results (\ref{HesseandLeopardi}) and (\ref{Hesse}). Since the $s$-energy is unbounded as two points approach each other, and since spherical designs can have points arbitrarily close together, the separation constraint is needed to guarantee any asymptotic bounds on the energy.

Denote by $\mathcal{E}_{d}^{(s)}(N)$ the minimal discrete $s$-energy for $N$-points on the sphere
\begin{equation}\label{minCoulomb}
 \mathcal{E}_{d}^{(s)}(N):=\inf\limits_{X_{N}}E_{d}^{(s)}(X_{N}),
\end{equation}
where the infimum is taken over all $N$-points subsets of $\mathbb{S}^{d}$.

Kuijlaars and Saff \cite{KuijlaarsSaff:1998Asymptotics} proved that for $d\geq2$ and $0<s<d$, there exists a constant $C_{d,s}>0$, such that
\begin{equation}\label{KuijlaarsSaff}
 \mathcal{E}_{d}^{(s)}(N)\leq \frac{1}{2} V_{d}(s)N^{2}-C_{d,s}N^{1+\frac{s}{d}},
\end{equation}
where $V_{d}(s)$ is the energy integral
\begin{align}\label{main_term}
  V_{d}(s):=\int\limits_{\mathbb{S}^{d}}\int\limits_{\mathbb{S}^{d}}\frac{1}{|\mathbf{x}-\mathbf{y}|^{s}}d\sigma_{d}(\mathbf{x})d\sigma_{d}(\mathbf{y})=
  \frac{\Gamma(\frac{d+1}{2})\Gamma(d-s)}{\Gamma(d-s+1)\Gamma(d-\frac{s}{2})}.
\end{align}

Earlier, Wagner \cite{Wagner:1992UpperBounds} had obtained the lower bounds
\begin{equation}\label{Wagner1}
 \mathcal{E}_{d}^{(s)}(N)\geq \frac{1}{2} V_{d}(s)N^{2}-C_{d,s}N^{1+\frac{s}{d}}, \ \ d-2<s<d,
\end{equation}
\begin{equation}\label{Wagner2}
 \mathcal{E}_{d}^{(s)}(N)\geq \frac{1}{2} V_{d}(s)N^{2}-C_{d,s}N^{1+\frac{s}{2+s}}, \ \ d\geq3, \ 0<s\leq d-2.
\end{equation}
The combination of (\ref{KuijlaarsSaff}) and (\ref{Wagner1}) leads to the correct order of
 $\mathcal{E}_{d}^{(s)}(N)-\frac{1}{2} V_{d}(s)N^{2}$ for $d-2<s<d$.
 
 We show that for every well-separated sequence of
 $N$-point spherical $t$-designs on 
   $\mathbb{S}^{d}$, $d\geq2$,   with $N\asymp t^{d}$ the following asymptotic equality holds
 \begin{align*}
 E_{d}^{(s)}(X_{N})= \frac{1}{2}  \frac{\Gamma(\frac{d+1}{2})\Gamma(d-s)}{\Gamma(d-s+1)\Gamma(d-\frac{s}{2})}N^{2}+\mathcal{O}\Big(N^{1+\frac{s}{d}}\Big).
\end{align*}

The structure of the paper is as follows.

 Section~\ref{mainResults}
 contains the statements of all theorems.
Here we analyze energy integrals (\ref{nonsingularInt}) and  (\ref{singularInt}) with regard to area-regular partitions of the sphere. In particular the cases, when $K_{d}$ is  the reproducing kernel of a reproducing kernel Hilbert space of continuous functions on the sphere or the Riesz s-energy, are considered.  Then we  make a comparison   with the estimates of respective discrete energy sums for  spherical $t$-designs and minimizing point sets for $s$-energy.

In Section \ref{prelim} we summarize necessary background information for orthogonal polynomials.

In Section~\ref{proofTheor} we give the proofs of the theorems from the 
Section~\ref{mainResults}.

 Section~\ref{proofLem} contains the proofs of some technical lemmas, which are needed to proof Theorem 1.

\section{Formulation of main results}
\label{mainResults}

\subsection{The $s$-energy of spherical designs on $\mathbb{S}^{d}$}
\label{sEnergy}

 By a spherical cap $S(\mathbf{x}; \varphi)$ of centre $\mathbf{x}$ and angular radius
 $\varphi$ we mean
 \begin{equation*}
 S(\mathbf{x}; \varphi):=\big\{\mathbf{y}\in \mathbb{S}^{d} \big| \langle\mathbf{x},\mathbf{y}\rangle\geq \cos\varphi \big\}.
 \end{equation*}
The normalized surface area of a spherical cap is given by 
\begin{equation}\label{capArea}
|S(\mathbf{x}; \varphi)|=\frac{\Gamma((d+1)/2)}{\sqrt{\pi}\Gamma(d/2)}
\int\limits_{\cos\varphi}^{1}(1-t^{2})^{\frac{d}{2}-1}dt
\asymp(1-\cos\varphi)^{\frac{d}{2}} \quad\text{as } \varphi\rightarrow 0.
\end{equation}

If condition (\ref{wellSeparat}) holds for a sequence $(X_{N})_{N}$, then any  spherical cap $S(\mathbf{x}; \alpha_{N})$, $\mathbf{x}\in \mathbb{S}^{d}$, where
\begin{equation}\label{alphaN}
\alpha_{N}:=\arccos\Big(1-\frac{c^{2}_{1}}{8N^{\frac{2}{d}}}\Big),
\end{equation}
 contains at
 most one point of the set $(X_{N})_{N}$. 

From the elementary estimates
\begin{equation}\label{sinIneq}
\sin\theta\leq \theta\leq \frac{\pi}{2}\sin\theta, \quad 0\leq\theta\leq \frac{\pi}{2},
\end{equation}
we obtain
\begin{equation}\label{alphaEstim}
\Big(1-\frac{c^{2}_{1}}{16N^{\frac{2}{d}}}\Big)^{\frac{1}{2}}\frac{c_{1}}{2N^{\frac{1}{d}}}\leq\alpha_{N}\leq 
\frac{\pi}{4}\Big(1-\frac{c^{2}_{1}}{16N^{\frac{2}{d}}}\Big)^{\frac{1}{2}}\frac{c_{1}}{N^{\frac{1}{d}}}.
 \end{equation} 

\begin{thm}\label{theoremRisz} Let $d\geq2$ be fixed, and
  $(X_{N(t)})_t$ be a sequence of well-separated spherical $t$-designs on
   $\mathbb{S}^{d}$, 
  $t$ and $N(t)$ satisfying relation (\ref{NandT}).
Then for the $s$-energy $E_{d}^{(s)}(X_{N})$ the following asymptotic equality holds
 \begin{equation}\label{theorem1}
 E_{d}^{(s)}(X_{N})= \frac{1}{2}  \frac{\Gamma(\frac{d+1}{2})\Gamma(d-s)}{\Gamma(d-s+1)\Gamma(d-\frac{s}{2})}N^{2}+\mathcal{O}\Big(N^{1+\frac{s}{d}}\Big).
\end{equation}
\end{thm}

\subsection{Estimates for energy integrals in the nonsingular case}
\label{nonsingular}

 We  consider area-regular partitions for which all regions $A_{i}$ have small diameters:
 $\mathrm{diam}(A_{i})\leq C N^{-\frac{1}{d}}$ for $i=1,...,N$. Here $C$ is a constant that does not depend on $N$ (see, e.g., \cite{GiganteLeopardi2017:Diameter}).

Let $\sigma_{j}^{*}$ be the restriction of the measure $N\sigma$ to $A_{i}$: $\sigma_{i}^{*}(\cdot)=\sigma(A_{i}\cap \cdot)N$. Then each $\sigma_{j}^{*}$ is a probability measure.

\begin{thm}\label{theorem_det}
Let $K_{d}$ be  a continuous function on $[-1,1]$, which is given by (\ref{kernel}).
Then there exists a positive constant $C_{d}$, such that, for the energy $E(K_{d}, X_{N})$ of the form (\ref{energy}) the following estimate holds
\begin{multline}\label{theoremDet}
\int\limits_{A_{1}}...\int\limits_{A_{N}}E(K_{d}, X_{N})d\sigma_{1}^{*}(\mathbf{x}_{1})...d\sigma_{N}^{*}(\mathbf{x}_{N})-a_{0}\\
\leq \frac{C_{d}}{N}\left(N^{-\frac{2}{d}}\sum\limits_{n=1}^{[N^{\frac{1}{d}}]}a_{n}n^{2}+\sum\limits_{n=[N^{\frac{1}{d}}]+1}^{\infty}a_{n} \right).
\end{multline}
\end{thm}

Let us apply the estimate (\ref{theoremDet}) for the reproducing kernels of Hilbert spaces and compare it with known estimates of worst-case error in these spaces. Before that, we need some additional background.

We denote by $\{Y_{\ell,k}^{(d)}: k=1,\ldots, Z(d,\ell)\}$ a collection of 
$\mathbb{L}_{2}(\sigma_{d})$-orthonormal real spherical harmonics (homogeneous harmonic  polynomials in $d+1$ variables  restricted to $\mathbb{S}^{d}$) of degree $\ell$ (see, e.g., \cite{Mueller1966:spherical_harmonics}), where
\begin{equation}\label{Zd}
  Z(d,0)=1,\quad Z(d,\ell)=(2\ell+d-1)
  \dfrac{\Gamma(\ell+d-1)}{\Gamma(d)\Gamma(\ell+1)}\sim
  \frac{2}{\Gamma(d)}\ell^{d-1}, \quad \ell\rightarrow\infty.
 \end{equation}
Each spherical harmonic  $Y_{\ell,k}^{(d)}$ of exact degree $\ell$ is an eigenfunction of the negative Laplace-Beltrami operator $-\Delta^{*}_{d}$ with eigenvalue
$ \lambda_{\ell}:=\ell(\ell+d-1) $.
 
The spherical harmonics of degree $\ell$ satisfy the addition theorem: 
\begin{equation}\label{additiontheorem}
 \sum\limits_{k=1}^{Z(d,\ell)}Y_{\ell,k}^{(d)}(\mathbf{x})Y_{\ell,k}^{(d)}(\mathbf{y})=Z(d,\ell)P_{\ell}^{(d)}(\langle\mathbf{x},\mathbf{y}\rangle).
 \end{equation}

The Sobolev space  $\mathbb{H}^{s}(\mathbb{S}^{d})$ 
 for $s\geq 0$ consists of all functions $f\in\mathbb{L}_{2}(\mathbb{S}^{d})$ 
with finite norm 
 \begin{equation}\label{normSobolev}
\|f\|_{\mathbb{H}^{s}}=\bigg(\sum\limits_{\ell=0}^{\infty}\sum\limits_{k=1}^{Z(d,\ell)}
\left(1+\lambda_{\ell}\right)^{s}|\hat{f}_{\ell,k}|^{2}\bigg)^{\frac{1}{2}}, 
\end{equation}
 where the Laplace-Fourier coefficients are given by the formula
\begin{equation*}
  \hat{f}_{\ell,k}:=(f,Y_{\ell,k}^{(d)})_{\mathbb{S}^{d}}=
  \int_{\mathbb{S}^{d}}f(\mathbf{x})Y_{\ell,k}^{(d)}(\mathbf{x})
  d\sigma_{d}(\mathbf{x}).
\end{equation*}

The worst-case (cubature) error of the equal weight numerical integration rule $Q[X_{N}]$ in a
Banach space $B$ of continuous functions on $\mathbb{S}^{d}$ with norm
$\|\cdot\|_{B}$ is defined by
 \begin{equation}\label{wce}
 \mathrm{wce}(Q[X_{N}];B):=\sup\limits_{f\in B,\|f\|_{B}\leq1}
 \left|\frac{1}{N}\sum\limits_{i=1}^{N}f(\mathbf{x}_{i})-\int_{\mathbb{S}^{d}}f(\mathbf{x})d\sigma_{d}(\mathbf{x})\right|.
 \end{equation}

The worst-case error for the Sobolev space $\mathbb{H}^{s}(\mathbb{S}^{d})$ can be expressed as (see, e.g., \cite{Brauchart-Saff-Sloan+2014:qmc_designs}) 
  \begin{equation}\label{wceKernelHs}
\mathrm{wce}(Q[X_{N}];\mathbb{H}^{s}(\mathbb{S}^{d}))^{2}=\frac{1}{N^{2}}\sum\limits_{i,j=1}^{N}\tilde{K}^{(s)}_{d}(\langle\mathbf{x}_{i},\mathbf{x}_{j}\rangle),
 \end{equation}
  where $\tilde{K}^{(s)}_{d}$ denotes the reproducing kernel Hilbert space  $\mathbb{H}^{s}(\mathbb{S}^{d})$, $s>\frac{d}{2}$, with the constant term removed
 \begin{equation}\label{kernelHs}
 \tilde{K}^{(s)}_{d}(\mathbf{x},\mathbf{y})=\sum\limits_{\ell=1}^{\infty}(1+\lambda_{\ell})^{-s}Z(d,\ell)
 P_{\ell}^{(d)}(\langle\mathbf{x},\mathbf{y}\rangle).
 \end{equation}

If $\tilde{K}^{(s)}_{d}(\mathbf{x},\mathbf{y})$ is given by (\ref{kernelHs}), then
\begin{multline}\label{probabilSobol}
\frac{1}{N^{2}}\int\limits_{A_{1}}...\int\limits_{A_{N}}\sum\limits_{i,j=1}^{N}\tilde{K}^{(s)}_{d}(\langle\textbf{x}_{i},\mathbf{x}_{j}\rangle)d\sigma_{1}^{*}(\mathbf{x}_{1})...d\sigma_{N}^{*}(\mathbf{x}_{N}) \\
\ll 
\begin{cases}
 N^{-\frac{2s}{d}}, & \text{if }
  \frac{d}{2}<s<1+\frac{d}{2}, \\
N^{-1-\frac{2}{d}}\ln N , & \text{if }  s=1+\frac{d}{2}, \\
N^{-1-\frac{2}{d}}, & \text{if }  s>1+\frac{d}{2}.
  \end{cases}
\end{multline}
Here and further we use the Vinogradov notation $a_{n}\ll b_{n}$  to mean that there exists positive constant $C$  independent of $n$ such that 
 $a_{n}\leq  C b_{n}$  for all $n$.

In \cite{Brauchart-Hesse2007:numerical_integration} it was proved that there exists $C_{d,s}>0$, such that for every $N$-point spherical $t$-design  $X_{N}$ on $\mathbb{S}^{d}$ with $N\asymp t^{d}$ 
\begin{align}\label{BrauchHesseUpper}
\mathrm{wce}(Q[X_{N}];\mathbb{H}^{s}(\mathbb{S}^{d}))^{2}\leq \frac{C_{s,d}}{N^{\frac{2s}{d}}}.
\end{align}

Let  the space $\mathbb{H}^{(\frac{d}{2},\gamma)}(\mathbb{S}^{d})$, $\gamma>\dfrac{1}{2}$, 
(see \cite{GrabnerStepanyuk2018}) be  the  set of all functions $f\in\mathbb{L}_{2}(\mathbb{S}^{d})$ with finite norm
\begin{align*}
\|f\|_{\mathbb{H}^{(\frac{d}{2},\gamma)}}^2:=\sum\limits_{\ell=0}^{\infty}
\left(1+\lambda_{\ell}\right)^{\frac{d}{2}}\left(\ln\left(3+\lambda_{\ell}\right)\right)^{2\gamma}\sum\limits_{k=1}^{Z(d,\ell)}|\hat{f}_{\ell,k}|^{2}<\infty.
\end{align*}

The worst-case error for the space $\mathbb{H}^{(\frac{d}{2},\gamma)}(\mathbb{S}^{d})$, $\gamma>\dfrac{1}{2}$, can be computed by the formula  
\begin{align*}
\mathrm{wce}(Q[X_{N}];\mathbb{H}^{(\frac{d}{2},\gamma)}(\mathbb{S}^{d}))^{2}=\frac{1}{N^{2}}\sum\limits_{i,j=1}^{N}\tilde{K}^{(\frac{d}{2},\gamma)}(\mathbf{x}_{i},\mathbf{x}_{j}),
\end{align*}
  where $\tilde{K}^{(\frac{d}{2},\gamma)}$ denotes the reproducing kernel of the Hilbert space
    $\mathbb{H}^{(\frac{d}{2},\gamma)}(\mathbb{S}^{d})$, ${\gamma>\frac{1}{2}}$, with the constant term removed
 \begin{equation}\label{kernelH}
 \tilde{K}^{(\frac{d}{2},\gamma)}(\mathbf{x},\mathbf{y})=\sum\limits_{\ell=1}^{\infty}\left(1+\lambda_{\ell}\right)^{-\frac{d}{2}}\left(\ln\left(2+\lambda_{\ell}\right)\right)^{-2\gamma}Z(d,\ell)
 P_{\ell}^{(d)}(\langle\mathbf{x},\mathbf{y}\rangle).
 \end{equation}

From (\ref{theoremDet}) we have that for  $\tilde{K}^{(\frac{d}{2},\gamma)}$, defined  by formula (\ref{kernelH}), the following estimate is true
\begin{align}\label{probabilH}
\frac{1}{N^{2}}\int\limits_{A_{1}}...\int\limits_{A_{N}}\sum\limits_{i,j=1}^{N}\tilde{K}^{(\frac{d}{2},\gamma)}(\langle\textbf{x}_{i},\mathbf{x}_{j}\rangle)d\sigma_{1}^{*}(\mathbf{x}_{1})...d\sigma_{N}^{*}(\mathbf{x}_{N})
\ll N^{-1}\left(\ln N\right)^{-2\gamma+1}.
\end{align}

 In \cite{GrabnerStepanyuk2018} it was proved that there exist constants  $C_{d,\gamma}^{(1)} $ and $C_{d,\gamma}^{(2)}$,  such that for every $N$-point well separated spherical $t$-design  $X_{N}$ on $\mathbb{S}^{d}$ 
\begin{align}\label{GS}
 C_{d,\gamma}^{(1)}N^{-1}\left(\ln N\right)^{-2\gamma+1}\leq
   \mathrm{wce}(Q[X_{N}];\mathbb{H}^{(\frac{d}{2},\gamma)}(\mathbb{S}^{d}))^{2}
   \leq C_{d,\gamma}^{(2)}N^{-1}\left(\ln N\right)^{-2\gamma+1}.
\end{align}

\subsection{Estimates for energy integrals in the singular case}
\label{singular}

In this subsection we consider the case of singular kernel, when in the energy (\ref{energy})  the diagonal terms are omitted. We denote it by
\begin{align}\label{energy0}
\tilde{E}(K_{d}, X_{N}):=\frac{1}{N^{2}}\sum\limits_{i,j=1, \atop i\neq j}^{N}K_{d}(\langle\mathbf{x}_{i},\mathbf{x}_{j}\rangle).
\end{align}

\begin{thm}\label{theorem_sing}
 Let $K_{d}$ is a continuous function on $[-1,1)$, $\lim\limits_{x\rightarrow1}K_{d}(x)=\infty$, $\int\limits_{\mathbb{S}^{d}}K_{d}(\langle\mathbf{x},\mathbf{y}\rangle)d\sigma(\mathbf{x})d\sigma(\mathbf{y})<\infty$ and there exist $c_{2}>0$ and $\mathbf{y}_{i}\in A_{i}$, such that  each region $A_{i}$ of an area regular partition  
$\{A_{i}\}_{i=1}^{N}$ contains a  spherical cap $S(\mathbf{y}_{i}; c_{2}N^{-\frac{1}{d}})$ in its interior.

Then
\begin{multline}\label{theoremSingular}
\int\limits_{A_{1}}...\int\limits_{A_{N}}\tilde{E}(K_{d}, X_{N})d\sigma_{1}^{*}(\mathbf{x}_{1})...d\sigma_{N}^{*}(\mathbf{x}_{N})\\
= a_{0}+\frac{1}{N}\mathcal{O}\left(
\int\limits_{\cos(c_{2}N^{-\frac{1}{d}})}^{1}K_{d}(t)(1-x^{2})^{\frac{d}{2}-1}dx +\max\limits_{-1\leq x\leq1-\frac{2c_{2}^{2}}{\pi^{2}}}K_{d}(x)\right).
\end{multline}
\end{thm}

The existence of an area regular partition  
$\{A_{i}\}_{i=1}^{N}$, such that  each region $A_{i}$  contains the  spherical cap $S(\mathbf{y}_{i}; c_{2}N^{-\frac{1}{d}})$ in its interior was shown by Gigante and Leopardi in \cite{GiganteLeopardi2017:Diameter}.

Let $K_{s,d}$ be the Riesz kernel: $K_{s,d}(x)=\frac{1}{2^{\frac{s}{2}}}(1-x)^{-\frac{s}{2}}$, $0<s<d$, then
\begin{multline}\label{integrRiesz}
\frac{1}{N}
\int\limits_{1-cN^{-\frac{2}{d}}}^{1}K_{s,d}(t)(1-t^{2})^{\frac{d}{2}-1}dt\ll  \frac{1}{N}
\int\limits_{\cos(c_{2}N^{-\frac{1}{d}})}^{1}(1-x)^{\frac{d}{2}-\frac{s}{2}-1}dx\\
\ll \frac{1}{N} (N^{-\frac{2}{d}})^{\frac{d}{2}-\frac{s}{2}}= N^{-2+\frac{s}{d}}
\end{multline}
and
\begin{align}\label{maxRi}
\frac{1}{N}\max\limits_{-1\leq x\leq1-\frac{2c_{2}^{2}}{\pi^{2}}}K_{s,d}(x)\ll N^{-1+\frac{s}{d}}.
\end{align}

Thus, we have that for the Riesz kernel $K_{s,d}$ the following estimate holds
\begin{multline}\label{RieszProb}
\int\limits_{A_{1}}...\int\limits_{A_{N}}\sum\limits_{i,j=1, \atop i\neq j}^{N}K(\langle\mathbf{x}_{i},\mathbf{x}_{j}\rangle)d\sigma_{1}^{*}(\mathbf{x}_{1})...d\sigma_{N}^{*}(\mathbf{x}_{N})) \\
=
a_{0}N^{2}+\mathcal{O}(N^{1+\frac{s}{d}})= \frac{\Gamma(\frac{d+1}{2})\Gamma(d-s)}{\Gamma(d-s+1)\Gamma(d-\frac{s}{2})}N^{2}+\mathcal{O}(N^{1+\frac{s}{d}}).
\end{multline}

\subsection{Comparison of the estimates for some probabilistic and deterministic point sets}
\label{comparison}

 Probabilistic models are often used to show existence of good point sets. But the comparison shows that in many cases $t$-designs give better bounds for the quality measure under consideration.

Indeed, on the basis of (\ref{probabilSobol}) and (\ref{BrauchHesseUpper}) we can summarize, that in the case ${\frac{d}{2}<s<1+\frac{d}{2}}$ spherical $t$-designs are as good as probabilistic point sets, and in case $s\geq\frac{d}{2}+1$, spherical $t$-designs give  better bounds for  $\mathrm{wce}(Q[X_{N}];\mathbb{H}^{s}(\mathbb{S}^{d}))$.

From (\ref{probabilH}) and (\ref{GS}) it follows, that  for the worst-case error \linebreak
 $ \mathrm{wce}(Q[X_{N}];\mathbb{H}^{(\frac{d}{2},\gamma)}(\mathbb{S}^{d}))$,  spherical $t$-designs are as good as probabilistic point sets.

Comparing formula (\ref{theorem1}) with (\ref{RieszProb}), we have that for the Riesz $s$-energy, $0<s<d$, well-separated $t$-designs are as good as probabilistic point sets.
Also according to relations  (\ref{KuijlaarsSaff}) and (\ref{Wagner1}), with respect to the order of the error term, well-separated $t$-designs and probabilistic point sets are as good  as point sets which  minimize the Riesz $s$ energy (in the case $d-2<s<d$).

\section{Preliminaries}
\label{prelim}

In this paper we use the Pochhammer symbol $(a)_{n}$, where
$n\in \mathbb{N}_{0}$ and $a\in \mathbb{R}$, defined by
\begin{equation*}
  (a)_{0}:=1, \quad (a)_{n}:=a(a+1)\ldots(a+n-1)\quad \mathrm{for} \quad
  n\in \mathbb{N},
\end{equation*}
which can be  written in the terms of the gamma function $\Gamma(z)$ by means of
\begin{equation}\label{Pochhammer}
 (a)_{\ell}=\frac{\Gamma(\ell+a)}{\Gamma(a)}.
 \end{equation}
 For fixed $a,b$ the following asymptotic equality is true
 \begin{equation}\label{gamma}
 \frac{\Gamma(n+a)}{\Gamma(n+b)}= n^{a-b}\Big(1+\mathcal{O}\Big(\frac{1}{n}\Big) \Big) \ \ \mathrm{as} \ \ n\rightarrow \infty.
\end{equation}

For any integrable function $f: [-1, 1]\rightarrow \mathbb{R}$ (see, e.g.,
\cite{Mueller1966:spherical_harmonics}) we have
\begin{equation}\label{a1}
 \int\limits_{\mathbb{S}^{d}}f(\langle\mathbf{x},\mathbf{y}\rangle)d\sigma_{d}(\mathbf{x})=\frac{\Gamma(\frac{d+1}{2})}{\sqrt{\pi}\Gamma(\frac{d}{2})}\int\limits_{-1}^{1}f(t)(1-t^{2})^{\frac{d}{2}-1}dt \quad \forall \mathbf{y}\in \mathbb{S}^{d}.
\end{equation}

The Jacobi polynomials $\mathcal{P}_{\ell}^{(\alpha,\beta)}(x)$ are the polynomials
orthogonal over the interval $[-1,1]$ with the weight function
$w_{\alpha,\beta}(x)=(1-x)^{\alpha}(1+x)^{\beta}$ and normalized by the
relation
\begin{equation}\label{JacobiMax}
  \mathcal{P}_{\ell}^{(\alpha,\beta)}(1)=\binom {\ell+\alpha}\ell=
  \frac{(1+\alpha)_{\ell}}{\ell!}\sim\frac{1}{\Gamma(1+\alpha)}\ell^{\alpha},
  \quad \alpha,\beta>-1.
 \end{equation}
 (see, e.g., \cite[(5.2.1)]{Magnus-Oberhettinger-Soni1966:formulas_theorems}).

 Notice that
 \begin{equation}\label{LegendreJacobi}
 P_{n}^{(d)}(x)=\frac{n!}{(d/2)_{n}}\mathcal{P}_{n}^{(\frac{d}{2}-1, \frac{d}{2}-1)}(x).
 \end{equation}

 For fixed ${\alpha, \beta>-1}$ and ${0< \theta<\pi}$, the following relation
 gives an asymptotic approximation for $\ell\rightarrow\infty$ (see,
 e.g.,\cite[Theorem 8.21.13]{Szegoe1975:orthogonal_polynomials})
\begin{multline*}
  \mathcal{P}_{\ell}^{(\alpha,\beta)}(\cos \theta)=\frac{1}{\sqrt{\pi}}\ell^{-1/2}
  \Big(\sin\frac{\theta}{2}\Big)^{-\alpha-1/2}
  \Big(\cos\frac{\theta}{2}\Big)^{-\beta-1/2}\\
  \times\Big\{\cos \Big(\Big(\ell+\frac{\alpha+\beta+1}{2}\Big)\theta-
  \frac{2\alpha+1}{4}\pi\Big)+\mathcal{O}(\ell\sin\theta)^{-1}\Big\}.
\end{multline*}
Thus, for
$c_{\alpha,\beta}\ell^{-1}\leq\theta\leq\pi-c_{\alpha,\beta}\ell^{-1}$ the last
asymptotic equality yields
 \begin{equation}\label{JacobiIneq}
   |\mathcal{P}_{\ell}^{(\alpha,\beta)}(\cos \theta)|\leq \tilde{c}_{\alpha,\beta}
   \ell^{-1/2}(\sin\theta)^{-\alpha-1/2}+
 \tilde{c}_{\alpha,\beta}\ell^{-3/2}(\sin\theta)^{-\alpha-3/2}, \quad\alpha\geq\beta.
 \end{equation}

The following differentiation formula holds
\begin{equation}\label{JacobiDifferen}
   \frac{d}{dx}\mathcal{P}_{n}^{(\alpha,\beta)}(x)=\frac{\alpha+\beta+n+1}{2}\mathcal{P}_{n-1}^{(\alpha+1,\beta+1)}(x).
 \end{equation}

If $\lambda>d-1$, $0<s<d$, (using formula  \cite[(5.3.4)]{Magnus-Oberhettinger-Soni1966:formulas_theorems}) and expressing 
the Gegenbauer polynomials  via Jacobi polynomials  (see, e.g.,
 \cite[(5.3.1)]{Magnus-Oberhettinger-Soni1966:formulas_theorems})),
we have that  for ${-1<x<1}$ the following expansion holds 
\begin{multline}\label{expansionGegenbauer1}
(1-x)^{-\frac{s}{2}}=2^{2\lambda-\frac{s}{2}}\pi^{-\frac{1}{2}}\Gamma(\lambda)\Gamma\Big(\lambda-\frac{s}{2}+\frac{1}{2}\Big) \\
\times\sum\limits_{n=0}^{\infty}\frac{(n+\lambda)(\frac{s}{2})_{n}}{\Gamma(n+2\lambda-\frac{s}{2}+1)}\frac{(2\lambda)_{n}}{(\lambda+\frac{1}{2})_{n}}\mathcal{P}_{n}^{(\lambda-\frac{1}{2},\lambda-\frac{1}{2})}(x).
\end{multline}

\section{Proof of Theorems 1-3}
\label{proofTheor}

\begin{proof}[Proof of Theorem  \ref{theoremRisz}]

 For each $i\in\{1,\ldots,N\}$ we divide the sphere $\mathbb{S}^{d}$ into an
  upper hemisphere $H_{i}^{+}$ with 'north pole' $\mathbf{x}_{i}$ and a lower
  hemisphere $H_{i}^{-}$:
\begin{equation*}
H_{i}^{+}:=\Big\{\mathbf{x}\in\mathbb{S}^{d}\Big|\langle\mathbf{x}_{i},\mathbf{x}\rangle\geq0 \Big\},
\end{equation*}
\begin{equation*}
H_{i}^{-}:=\mathbb{S}^{d}\setminus H_{i}^{+}.
\end{equation*}
We split the $s$-energy into two parts
\begin{equation}\label{1split}
E_{d}^{(s)}(X_{N})=\frac{1}{2}\sum\limits_{j=1}^{N}
  {\mathop{\sum}\limits_{i=1,\atop \textbf{x}_{i}\in H^{\pm}_{i}\setminus S(\pm x_{j};\alpha_{N})}^{N}}|\mathbf{x}_{i}-\mathbf{x}_{j}|^{-s}+
  \frac{1}{2}\sum\limits_{j=1}^{N}
  {\mathop{\sum}\limits_{i=1,\atop \textbf{x}_{i}\in  S(- x_{j};\alpha_{N})}^{N}}|\mathbf{x}_{i}-\mathbf{x}_{j}|^{-s}.
\end{equation}

From (\ref{wellSeparat}) and the fact the spherical cap $S(- \mathbf{x}_{j};\alpha_{N})$ contains at most one point of $X_{N}$, the second term in (\ref{1split}), where the scalar product is close to $-1$, can be bounded from above by
\begin{equation}\label{estim1split}
  \frac{1}{2}\sum\limits_{j=1}^{N}
  {\mathop{\sum}\limits_{i=1,\atop \textbf{x}_{i}\in  S(- x_{j};\alpha_{N})}^{N}}|\mathbf{x}_{i}-\mathbf{x}_{j}|^{-s}<\frac{1}{2}N\Big(4-\frac{c_{1}^{2}}{4N^{\frac{2}{d}}} \Big)^{-\frac{s}{2}}<\frac{1}{2}N.
\end{equation}

Noting that
\begin{equation}\label{distance}
|\mathbf{x}_{i}-\mathbf{x}_{j}|^{-1}=\frac{1}{\sqrt{2}}(1-\langle\textbf{x}_{i},\mathbf{x}_{j}\rangle)^{-\frac{1}{2}},
\end{equation}
taking into account that the Jacobi  series (\ref{expansionGegenbauer1}) converges uniformly in \linebreak ${\Big[-1+\frac{c^{2}_{1}}{8N^{\frac{2}{d}}},1-\frac{c^{2}_{1}}{8N^{\frac{2}{d}}} \Big]}$,
 and substituting 
$\lambda=\frac{d}{2}+K+\frac{1}{2}$, $K>\frac{d}{2}+1$ in the expansion (\ref{expansionGegenbauer1}),  we get that 
\begin{multline}\label{expansionSubst}
\frac{1}{2}\sum\limits_{j=1}^{N}
  {\mathop{\sum}\limits_{i=1,\atop \textbf{x}_{i}\in H^{\pm}_{i}\setminus S(\pm x_{j};\alpha_{N})}^{N}}|\mathbf{x}_{i}-\mathbf{x}_{j}|^{-s}=
  \frac{1}{2^{1+\frac{s}{2}}}\sum\limits_{j=1}^{N}
  {\mathop{\sum}\limits_{i=1,\atop \textbf{x}_{i}\in H^{\pm}_{i}\setminus S(\pm x_{j};\alpha_{N})}^{N}}(1-\langle\textbf{x}_{i},\mathbf{x}_{j}\rangle)^{-\frac{s}{2}} \\
  =\frac{1}{2}E_{h_{t}}(X)+\frac{1}{2}E_{r_{t}}(X),
\end{multline}
where
\begin{multline}\label{s_tDefinition}
h_{t}(x)=h_{t}(s,d,K,t,x):=2^{d+2K-s+1}\pi^{-\frac{1}{2}}\Gamma\Big(\frac{d}{2}+K+\frac{1}{2}\Big)\Gamma\Big(\frac{d}{2}+K-\frac{s}{2}+1\Big) \\
\times\sum\limits_{n=0}^{t}\frac{(n+\frac{d}{2}+K+\frac{1}{2})(\frac{s}{2})_{n}}{\Gamma(n+d+2K-\frac{s}{2}+2 )}\frac{(d+2K+1)_{n}}{(\frac{d}{2}+K+1)_{n}}
\mathcal{P}_{n}^{(\frac{d}{2}+K, \ \frac{d}{2}+K)}(x),
\end{multline}
\begin{multline}\label{r_tDefinition}
r_{t}(x)=r_{t}(s,d,K,t,x):=2^{d+2K-s+1}\pi^{-\frac{1}{2}}\Gamma\Big(\frac{d}{2}+K+\frac{1}{2}\Big)\Gamma\Big(\frac{d}{2}+K-\frac{s}{2}+1\Big) \\
\times\sum\limits_{n=t+1}^{\infty}\frac{(n+\frac{d}{2}+K+\frac{1}{2})(\frac{s}{2})_{n}}{\Gamma(n+d+2K-\frac{s}{2}+2 )}\frac{(d+2K+1)_{n}}{(\frac{d}{2}+K+1)_{n}}
\mathcal{P}_{n}^{(\frac{d}{2}+K, \ \frac{d}{2}+K)}(x),
\end{multline}
and 
\begin{align}\label{energyParts}
E_{U}(X):=\sum\limits_{j=1}^{N}
  {\mathop{\sum}\limits_{i=1,\atop \textbf{x}_{i}\in H^{\pm}_{i}\setminus S(\pm x_{j};\alpha_{N})}^{N}}U(\langle\textbf{x}_{i},\mathbf{x}_{j}\rangle).
\end{align}

To finish the proof we will need following two  lemmas. We postpone the proof of lemmas to the next section.

\begin{lemma}\label{lem1} Let $d\geq2$  
  be fixed and   $N\asymp t^{d}$.  Then for any $K>\frac{d}{2}$,
  $K\in \mathbb{N}$, and $0<s<d$ there exists positive constant $C_{d,s} $, such 
  that
\begin{align}\label{Lemma1}
E_{r_{t}}(X)\leq C_{d,s} N^{1+\frac{s}{d}}.
\end{align}
\end{lemma}

\begin{lemma}\label{lem2} Let $d\geq2$  
  be fixed, let $(X_{N(t)})_{t}$ be a sequence of well-separated $t$-designs and$N\asymp t^{d}$.  Then for any $0<s<d$ and $K>\frac{d}{2}$,
  $K\in \mathbb{N}$, the following asymptotic equality holds
\begin{align}\label{Lemma2}
E_{h_{t}}(X)=   \frac{\Gamma(\frac{d+1}{2})\Gamma(d-s)}{\Gamma(d-s+1)\Gamma(d-\frac{s}{2})}N^{2}+\mathcal{O}(Nt^{s}).
\end{align}
\end{lemma}

Formulas (\ref{1split})-(\ref{expansionSubst}), (\ref{Lemma1}) and (\ref{Lemma2}) yield (\ref{theorem1}).
Theorem \ref{theoremRisz} is proved.
\end{proof}

\begin{proof}[Proof of Theorem  \ref{theorem_det}]
Integrating $K_{d}$ with respect to the probability measure $d\sigma_{1}^{*}(\mathbf{x}_{1})...d\sigma_{N}^{*}(\mathbf{x}_{N})$, we obtain
\begin{multline}\label{integr}
\frac{1}{N^{2}}\int\limits_{A_{1}}...\int\limits_{A_{N}}\sum\limits_{i,j=1}^{N}K_{d}(\langle\textbf{x}_{i},\mathbf{x}_{j}\rangle)d\sigma_{1}^{*}(\mathbf{x}_{1})...d\sigma_{N}^{*}(\mathbf{x}_{N})\\
=\frac{1}{N}K_{d}(1)+\frac{1}{N^{2}}\sum\limits_{i,j=1, \atop i\neq j}^{N}\int\limits_{A_{i}}\int\limits_{A_{j}} K_{d}(\langle\textbf{x},\mathbf{y}\rangle)d\sigma_{i}^{*}(\mathbf{x})d\sigma_{j}^{*}(\mathbf{y}) \\
=\frac{1}{N}K_{d}(1)+\int\limits_{\mathbb{S}^{d}}\int\limits_{\mathbb{S}^{d}}K_{d}(\langle\textbf{x},\mathbf{y}\rangle)d\sigma(\mathbf{x})d\sigma(\mathbf{y})-
\frac{1}{N^{2}}\sum\limits_{i}^{N}\int\limits_{A_{i}}\int\limits_{A_{i}} K_{d}(\langle\textbf{x},\mathbf{y}\rangle)d\sigma_{i}^{*}(\mathbf{x})d\sigma_{i}^{*}(\mathbf{y}) \\
=\frac{1}{N}K_{d}(1)+a_{0}-
\frac{1}{N^{2}}\sum\limits_{i=1}^{N}\int\limits_{A_{i}}\int\limits_{A_{i}} K_{d}(\langle\textbf{x},\mathbf{y}\rangle)d\sigma_{i}^{*}(\mathbf{x})d\sigma_{i}^{*}(\mathbf{y}).
\end{multline}

Substituting (\ref{kernel}), we have 
\begin{multline}\label{dif1}
\frac{1}{N}K_{d}(1)-
\frac{1}{N^{2}}\sum\limits_{i=1}^{N}\int\limits_{A_{i}}\int\limits_{A_{i}} K_{d}(\langle\textbf{x},\mathbf{y}\rangle)d\sigma_{i}^{*}(\mathbf{x})d\sigma_{i}^{*}(\mathbf{y}) \\
=\frac{1}{N^{2}}\sum\limits_{i=1}^{N}\Big(\sum\limits_{n=0}^{\infty}a_{n}P_{n}^{(d)}(1)-
\sum\limits_{n=0}^{\infty}a_{n}P_{n}^{(d)}(\cos\theta_{i}) \Big) \\
=\frac{1}{N^{2}}\sum\limits_{i=1}^{N}\sum\limits_{n=0}^{[N^{\frac{1}{d}}]}a_{n}\big(P_{n}^{(d)}(1)-P_{n}^{(d)}(\cos\theta_{i}) \big)+
\frac{1}{N^{2}}\sum\limits_{i=1}^{N}\sum\limits_{n=[N^{\frac{1}{d}}]+1}^{\infty}a_{n}\big(P_{n}^{(d)}(1)-P_{n}^{(d)}(\cos\theta_{i}) \big),
\end{multline}
where $\cos\theta_{i}\in M_{i}$, $M_{i}:=\Big\{\langle\textbf{x},\mathbf{y}\rangle \ \Big| \ \textbf{x},\textbf{y}\in A_{i} \Big\}$.

The second term in right-hand side of (\ref{dif1}) can be bounded above by
\begin{align}\label{secondTerm}
\frac{1}{N^{2}}\sum\limits_{i=1}^{N}\sum\limits_{n=[N^{\frac{1}{d}}]+1}^{\infty}a_{n}\big(P_{n}^{(d)}(1)-P_{n}^{(d)}(\cos\theta_{i}) \big)\leq 
\frac{2}{N}\sum\limits_{n=[N^{\frac{1}{d}}]+1}^{\infty}a_{n}.
\end{align}

If $\mathrm{diam}(A_{i})\leq C N^{-\frac{1}{d}}$, then 
\begin{align*}
|\textbf{x}-\textbf{y}|\leq C N^{-\frac{1}{d}} \ \ \forall \textbf{x},\textbf{y}\in A_{i},
\end{align*}
and
\begin{align}\label{diam}
\cos\theta_{i}\geq 1- \frac{C^{2}}{2}N^{-\frac{2}{d}}.
\end{align}

Using the mean value theorem and relations (\ref{LegendreJacobi}), (\ref{JacobiDifferen}), and (\ref{JacobiMax}), we obtain that
\begin{multline}\label{dif2}
\frac{1}{N^{2}}\sum\limits_{i=1}^{N}\sum\limits_{n=1}^{[N^{\frac{1}{d}}]}a_{n}\big(P_{n}^{(d)}(1)-P_{n}^{(d)}(\cos\theta_{i}) \big)\\
=
\frac{1}{N^{2}}\sum\limits_{i=1}^{N}\sum\limits_{n=0}^{[N^{\frac{1}{d}}]}a_{n}\frac{n!}{(d/2)_{n}}\big(\mathcal{P}_{n}^{(\frac{d}{2}-1, \frac{d}{2}-1)}(1)-\mathcal{P}_{n}^{(\frac{d}{2}-1, \frac{d}{2}-1)}(\cos\theta_{i}) \big)\\
=
\frac{1}{N^{2}}\sum\limits_{i=1}^{N}\sum\limits_{n=0}^{[N^{\frac{1}{d}}]}a_{n}\frac{n!}{(d/2)_{n}}\big(1-\cos\theta_{i}\big)\frac{d}{dx}\mathcal{P}_{n-1}^{(\frac{d}{2}-1, \frac{d}{2}-1)}(\xi_{i}) \\
\leq \frac{1}{N}\sum\limits_{n=1}^{[N^{\frac{1}{d}}]}a_{n}\frac{n!}{(d/2)_{n}}\frac{d+n-1}{2}\big(1-\cos\theta_{i}\big)\mathcal{P}_{n-1}^{(\frac{d}{2}, \frac{d}{2})}(1)
 \ll \frac{1}{N^{1+\frac{2}{d}}}\sum\limits_{n=1}^{[N^{\frac{1}{d}}]}a_{n}n^{2},
\end{multline}
where $\xi_{i}\in[\cos\theta_{i}, 1]$.

Then Theorem~\ref{theorem_det} is proved.
\end{proof}

\begin{proof}[Proof of Theorem  \ref{theorem_sing}]
Integrating $\tilde{E}(K_{d}, X_{N})$ with respect to the probability measure $d\sigma_{1}^{*}(\mathbf{x}_{1})...d\sigma_{N}^{*}(\mathbf{x}_{N})$, we obtain
\begin{multline}\label{integrSingular}
\frac{1}{N^{2}}\int\limits_{A_{1}}...\int\limits_{A_{N}}\sum\limits_{i,j=1, \atop i\neq j}^{N}K_{d}(\langle\textbf{x}_{i},\mathbf{x}_{j}\rangle)d\sigma_{1}^{*}(\mathbf{x}_{1})...d\sigma_{N}^{*}(\mathbf{x}_{N})\\
=\frac{1}{N^{2}}\sum\limits_{i,j=1, \atop i\neq j}^{N}\int\limits_{A_{i}}\int\limits_{A_{j}} K_{d}(\langle\textbf{x},\mathbf{y}\rangle)d\sigma_{i}^{*}(\mathbf{x})d\sigma_{j}^{*}(\mathbf{y}) \\
=\int\limits_{\mathbb{S}^{d}}\int\limits_{\mathbb{S}^{d}}K_{d}(\langle\textbf{x},\mathbf{y}\rangle)d\sigma(\mathbf{x})d\sigma(\mathbf{y})-
\frac{1}{N^{2}}\sum\limits_{i=1}^{N}\int\limits_{A_{i}}\int\limits_{A_{i}} K_{d}(\langle\textbf{x},\mathbf{y}\rangle)d\sigma_{i}^{*}(\mathbf{x})d\sigma_{i}^{*}(\mathbf{y}) \\
=a_{0}-
\frac{1}{N^{2}}\sum\limits_{i=1}^{N}\int\limits_{A_{i}}\int\limits_{A_{i}} K_{d}(\langle\textbf{x},\mathbf{y}\rangle)d\sigma_{i}^{*}(\mathbf{x})d\sigma_{i}^{*}(\mathbf{y}).
\end{multline}

Taking into account,  that each $A_{i}$ contains the spherical cap $S(\mathbf{y}_{i}; c_{2}N^{-\frac{1}{d}})$ in its interior, we obtain
\begin{multline}\label{integrSingular1}
\frac{1}{N^{2}}\sum\limits_{i=1}^{N}\int\limits_{A_{i}}\int\limits_{A_{i}} K_{d}(\langle\textbf{x},\mathbf{y}\rangle)d\sigma_{i}^{*}(\mathbf{x})d\sigma_{i}^{*}(\mathbf{y}) \\
=
\frac{1}{N^{2}}\sum\limits_{i=1}^{N}\int\limits_{A_{i}}
\int\limits_{S(\mathbf{y}_{i}; c_{2}N^{-\frac{1}{d}})} K_{d}(\langle\textbf{x},\mathbf{y}\rangle)d\sigma_{i}^{*}(\mathbf{x})d\sigma_{i}^{*}(\mathbf{y})+\frac{1}{N}\mathcal{O}
\Big(\max\limits_{-1\leq x\leq1-\frac{2c_{2}^{2}}{\pi^{2}}}K_{d}(x)\Big).
\end{multline}

Using formula (\ref{a1}), we have that
\begin{multline}\label{integrSingular2}
\frac{1}{N^{2}}\sum\limits_{i=1}^{N}\int\limits_{A_{i}}
\int\limits_{S(\mathbf{y}_{i}; c_{2}N^{-\frac{1}{d}})} K_{d}(\langle\textbf{x},\mathbf{y}\rangle)d\sigma_{i}^{*}(\mathbf{x})d\sigma_{i}^{*}(\mathbf{y})\\
=\frac{1}{N}
\frac{\Gamma(\frac{d+1}{2})}{\sqrt{\pi}\Gamma(\frac{d}{2})}\int\limits_{\cos(c_{2}N^{-\frac{1}{d}})}^{1}K_{d}(x)(1-x^{2})^{\frac{d}{2}-1}dx.
\end{multline}

Formulas (\ref{integrSingular})-(\ref{integrSingular2}) imply (\ref{theoremSingular}). Theorem~\ref{theorem_sing} is proved.
\end{proof}

\section{Proof of Lemmas 1-2}
\label{proofLem}

\begin{proof}[Proof of Lemma  \ref{lem1}]

Applying relations (\ref{r_tDefinition}), (\ref{Pochhammer}), (\ref{gamma}) and (\ref{JacobiIneq}), we find that for $0<\theta<\pi$,
 \begin{multline}\label{r_ineq1}
|r_{t}(\cos\theta)|\ll
\sum\limits_{n=t+1}^{\infty}\frac{(n+\frac{d}{2}+K+\frac{1}{2})(\frac{s}{2})_{n}}{\Gamma(n+d+2K-\frac{s}{2}+2 )}\frac{(d+2K+1)_{n}}{(\frac{d}{2}+K+1)_{n}}\big|\mathcal{P}_{n}^{(\frac{d}{2}+K, \ \frac{d}{2}+K)}(\cos\theta)\Big| \\
\ll
\sum\limits_{n=t+1}^{\infty}n^{-\frac{d}{2}-K+s-1}\Big|\mathcal{P}_{n}^{(\frac{d}{2}+K, \ \frac{d}{2}+K)}(\cos\theta)\Big| \\
\ll
\sum\limits_{n=t+1}^{\infty}n^{-\frac{d}{2}-K+s-1}\Big(n^{-\frac{1}{2}}(\sin\theta)^{-\frac{d}{2}-K-\frac{1}{2}}+ n^{-\frac{3}{2}}(\sin\theta)^{-\frac{d}{2}-K-\frac{3}{2}}\Big) \\
\ll t^{-\frac{d}{2}-K+s-\frac{1}{2}}(\sin\theta)^{-\frac{d}{2}-K-\frac{1}{2}}+ t^{-\frac{d}{2}-K+s-\frac{3}{2}}(\sin\theta)^{-\frac{d}{2}-K-\frac{3}{2}}.
\end{multline}

We define $\theta_{ij}^{\pm}\in[0,\pi]$ by $\cos \theta_{ij}^{\pm}:=\langle\mathbf{x}_{i},\pm\mathbf{x}_{j}\rangle$. Then $\sin \theta_{ij}^{+}=\sin \theta_{ij}^{-}$.

This and formula (\ref{r_ineq1}) imply
\begin{multline}\label{r_ineq2}
E_{r_{t}}(X)\ll t^{-\frac{d}{2}-K+s-\frac{1}{2}}
\sum\limits_{j=1}^{N}
  {\mathop{\sum}\limits_{i=1,\atop \textbf{x}_{i}\in H^{\pm}_{i}\setminus S(\pm x_{j};\alpha_{N})}^{N}}(\sin\theta_{ij}^{\pm})^{-\frac{d}{2}-K-\frac{1}{2}} \\
  + 
  t^{-\frac{d}{2}-K+s-\frac{3}{2}}\sum\limits_{j=1}^{N}
  {\mathop{\sum}\limits_{i=1,\atop \textbf{x}_{i}\in H^{\pm}_{i}\setminus S(\pm x_{j};\alpha_{N})}^{N}}(\sin\theta_{ij}^{\pm})^{-\frac{d}{2}-K-\frac{3}{2}}.
\end{multline}

 From \cite[(3.30) and (3.33)]{Brauchart-Hesse2007:numerical_integration}, it
 follows that
 \begin{multline} \label{BrauchartHesse}
  \frac{1}{N^{2}}\sum\limits_{j=1}^{N}{\mathop{\sum}\limits^{N}_{
 i=1,\atop \mathbf{x}_{i}\in H_{j}^{\pm}\setminus
S(\pm\mathbf{x}_{j}; \frac{c}{n})}}
  (\sin\theta_{ij}^{\pm})^{-\frac{d}{2}+\frac{1}{2}-k-L}
\\
 \ll 1+n^{L+k-(d+1)/2}, \quad k=0,1,\ldots \quad \text{for }L>\frac{d+1}{2}.
 \end{multline}
Choosing $K>\frac{d+1}{2}$ and applying estimates (\ref{NandT}), (\ref{alphaEstim}) and (\ref{BrauchartHesse}) to each  term from the right part of (\ref{r_ineq2}), we have that
\begin{multline}\label{r_ineq3}
E_{r_{t}}(X)\ll t^{-\frac{d}{2}-K+s-\frac{1}{2}}
N^{2}(N^{\frac{1}{d}})^{K-\frac{d}{2}+\frac{1}{2}} 
  + 
  t^{-\frac{d}{2}-K+s-\frac{3}{2}}N^{2}(N^{\frac{1}{d}})^{K-\frac{d}{2}+\frac{3}{2}} \\
  \ll N^{1+\frac{s}{d}}.
\end{multline}
From (\ref{r_ineq3}) we get (\ref{Lemma1}). This completes the proof. 
\end{proof}

\begin{proof}[Proof of Lemma  \ref{lem2}]
The polynomial $h_{t}$ is a spherical polynomial of degree $t$ and $X_{N}$ is a spherical $t$-design. Thus, $h_{t}$ is integrating exactly by an equal weight integration rule with nodes from $X_{N}$, and 
\begin{multline}\label{s_eq1}
E_{h_{t}}(X)=\sum\limits_{j=1}^{N}
  {\mathop{\sum}\limits_{i=1,\atop \textbf{x}_{i}\in H^{\pm}_{i}\setminus S(\pm \mathbf{x}_{j};\alpha_{N})}^{N}}h_{t}(\langle\textbf{x}_{i},\mathbf{x}_{j}\rangle) \\
  =\sum\limits_{i,j=1}^{N}h_{t}(\langle\textbf{x}_{i},\mathbf{x}_{j}\rangle)-
  \sum\limits_{j=1}^{N}
  {\mathop{\sum}\limits_{i=1,\atop \textbf{x}_{i}\in S(- x_{j};\alpha_{N})}^{N}}h_{t}(\langle\textbf{x}_{i},\mathbf{x}_{j}\rangle)-Nh_{t}(1) \\
 =N^{2}\int\limits_{\mathbb{S}^{d}}h_{t}(\langle\textbf{x}_{i},\mathbf{x}\rangle)d\sigma_{d}(\textbf{x})-
  \sum\limits_{j=1}^{N}
  {\mathop{\sum}\limits_{i=1,\atop \mathbf{x}_{i}\in S(- x_{j};\alpha_{N})}^{N}}h_{t}(\langle\mathbf{x}_{i},\mathbf{x}_{j}\rangle)-Nh_{t}(1).
\end{multline}

We observe, that
\begin{align}\label{ineq1}
\Big| \sum\limits_{j=1}^{N}
  {\mathop{\sum}\limits_{i=1,\atop \mathbf{x}_{i}\in S(- x_{j};\alpha_{N})}^{N}}h_{t}(\langle\mathbf{x}_{i},\mathbf{x}_{j}\rangle)\Big|\leq N h_{t}(1). 
\end{align}

From relations  (\ref{Pochhammer}), (\ref{gamma}), (\ref{JacobiMax}) and (\ref{s_tDefinition})
\begin{multline}\label{abs_valS_t}
h_{t}(1) 
= 2^{d+2K-s+1}\pi^{-\frac{1}{2}}\Gamma\Big(\frac{d}{2}+K+\frac{1}{2}\Big)\Gamma\Big(\frac{d}{2}+K-\frac{s}{2}+1\Big) \\
\times\sum\limits_{n=0}^{t}\frac{(n+\frac{d}{2}+K+\frac{1}{2})(\frac{s}{2})_{n}}{\Gamma(n+d+2K-\frac{s}{2}+2 )}\frac{(d+2K+1)_{n}}{(\frac{d}{2}+K+1)_{n}}\mathcal{P}_{n}^{(\frac{d}{2}+K, \ \frac{d}{2}+K)}(1) \\
= 2^{d+2K-s+1}\pi^{-\frac{1}{2}}\Gamma\Big(\frac{d}{2}+K+\frac{1}{2}\Big)\Gamma\Big(\frac{d}{2}+K-\frac{s}{2}+1\Big) \\
\times\sum\limits_{n=0}^{t}\frac{(n+\frac{d}{2}+K+\frac{1}{2})(\frac{s}{2})_{n}}{\Gamma(n+d+2K-\frac{s}{2}+2 )}\frac{(d+2K+1)_{n}}{(\frac{d}{2}+K+1)_{n}}\frac{\Gamma(n+\frac{d}{2}+K+1)}{\Gamma(\frac{d}{2}+K+1)\Gamma(n+1)}\ll  t^{s}.
\end{multline}
Thus, relations (\ref{s_eq1})-\ref{abs_valS_t}) yield
\begin{align}\label{ineq_s_t1}
E_{h_{t}}(X)=N^{2}\int\limits_{\mathbb{S}^{d}}h_{t}(\langle\mathbf{x},\mathbf{y}\rangle)d\sigma_{d}(\textbf{x})+\mathcal{O}(Nt^{s}), \ \ \mathbf{y}\in \mathbb{S}^{d}.
\end{align}

The expansion (\ref{expansionGegenbauer1}) holds only inside the interval $(-1,1)$.
Thus, we  write the integral from (\ref{ineq_s_t1}) in the following way
\begin{multline}\label{s_eq2}
N^{2}\int\limits_{\mathbb{S}^{d}}h_{t}(\langle\mathbf{x},\mathbf{y}\rangle)d\sigma_{d}(\textbf{x}) \\
=
N^{2}\int\limits_{\mathbb{S}^{d}\setminus S(\pm\mathbf{y};\alpha_{N})}h_{t}(\langle\mathbf{x},\mathbf{y}\rangle)d\sigma_{d}(\textbf{x}) +
N^{2}\int\limits_{S(\pm\mathbf{y};\alpha_{N})}h_{t}(\langle\mathbf{x},\mathbf{y}\rangle)d\sigma_{d}(\textbf{x}) \\
=
N^{2}\int\limits_{\mathbb{S}^{d}\setminus S(\pm \mathbf{y};\alpha_{N})}\Big( 1- \langle\mathbf{x},\mathbf{y}\rangle)^{-\frac{s}{2}}-r_{t}(\langle\mathbf{x},\mathbf{y}\rangle) \Big) d\sigma_{d}(\textbf{x})
+
N^{2}\int\limits_{S(\pm \mathbf{y};\alpha_{N})}h_{t}(\langle\mathbf{x},\mathbf{y}\rangle)d\sigma_{d}(\textbf{x}) \\
= 
N^{2}\int\limits_{\mathbb{S}^{d}}\Big( 1- \langle\textbf{x},\mathbf{y}\rangle)^{-\frac{s}{2}} d\sigma_{d}(\textbf{x})+W_{t}(X_{N}), \ \ \mathbf{y}\in \mathbb{S}^{d},
\end{multline}
where we have used the fact, that the series $r_{t}(\langle\mathbf{x},\mathbf{y}\rangle)$ converges uniformly for all $\mathbf{x}\in \mathbb{S}^{d}\setminus S(\pm\mathbf{y};\alpha_{N})$, 
and 
\begin{multline}\label{W_def}
W_{t}(X_{N})=W_{t}(d,s,X_{N}):=
-N^{2}\int\limits_{S(\pm \mathbf{y};\alpha_{N})}(1- \langle\mathbf{x},\mathbf{y}\rangle)^{-\frac{s}{2}} d\sigma_{d}(\mathbf{x}) \\
-N^{2}\int\limits_{\mathbb{S}^{d}\setminus S(\pm \mathbf{y};\alpha_{N})}r_{t}(\langle\mathbf{x},\mathbf{y}\rangle)d\sigma_{d}(\mathbf{x})
+
N^{2}\int\limits_{S(\pm \mathbf{y};\alpha_{N})}h_{t}(\langle\mathbf{x},\mathbf{y}\rangle)d\sigma_{d}(\mathbf{x}). 
\end{multline}

Now let us show that 
\begin{align}\label{W_estim}
|W_{t}(X_{N})|\ll N^{1+\frac{s}{d}}.
\end{align}

For the  third term in (\ref{W_def}) the following estimate holds
\begin{multline}\label{ineq2}
\Big|N^{2}\int\limits_{S(\pm\mathbf{y};\alpha_{N})}h_{t}(\langle\mathbf{x},\mathbf{y}\rangle)d\sigma_{d}(\textbf{x})\Big|\leq N^{2}h_{t}(1)|S(\textbf{y};\alpha_{N})| \ll
N^{2}t^{s}|S(\textbf{y};\alpha_{N})| \\
\asymp N^{2}t^{s}(1-\cos\alpha_{N})^{\frac{d}{2}}\asymp N^{2}t^{s}(N^{-\frac{2}{d}})^{\frac{d}{2}}=Nt^{s},
\end{multline}
where we have used the formula for the normalized surface area of spherical cap (\ref{capArea}) and the estimates (\ref{abs_valS_t}) and (\ref{alphaEstim}).

Now we show, that
\begin{align}\label{firstTerm}
\Big|N^{2}\int\limits_{S(\pm \mathbf{y};\alpha_{N})}(1- \langle\mathbf{x},\mathbf{y}\rangle)^{-\frac{s}{2}} d\sigma_{d}(\textbf{x})\Big|\ll N^{1+\frac{s}{d}}.
\end{align}

Clearly,
\begin{align}\label{ineq3}
\Big|N^{2}\int\limits_{S(-\mathbf{y};\alpha_{N})}(1- \langle\mathbf{x},\mathbf{y}\rangle)^{-\frac{s}{2}} d\sigma_{d}(\mathbf{x})\Big|\ll N^{2}|S(\mathbf{x};\alpha_{N})|\ll N.
\end{align}

From (\ref{a1}) we have 
\begin{multline}\label{ineq4}
\Big|N^{2}\int\limits_{S(\mathbf{y};\alpha_{N})}(1- \langle\mathbf{x},\mathbf{y}\rangle)^{-\frac{s}{2}} d\sigma_{d}(\mathbf{x})\Big| \\
=N^{2}\frac{\Gamma(\frac{d+1}{2})}{\sqrt{\pi}\Gamma(\frac{d}{2})}\int\limits_{1-\frac{c_{1}^{2}}{8N^{\frac{2}{d}}}}^{1}(1-x)^{-\frac{s}{2}}(1-x^{2})^{\frac{d}{2}-1}dx 
\ll N^{1+\frac{s}{d}}.
\end{multline}
Combining (\ref{ineq3}) and (\ref{ineq4}), we obtain (\ref{firstTerm}). 

It remains to examine the  second term in (\ref{W_def}). Relation (\ref{a1}) and the estimate (\ref{r_ineq1}) allow us to write
\begin{multline}\label{ineq5}
\Big|N^{2}\int\limits_{\mathbb{S}^{d}\setminus S(\pm \mathbf{y};\alpha_{N})}r_{t}(\langle\mathbf{x},\mathbf{y}\rangle)d\sigma_{d}(\mathbf{x})
\Big| \ll
N^{2}\int\limits_{-1+\frac{c_{1}^{2}}{8N^{\frac{2}{d}}}}^{1-\frac{c_{1}^{2}}{8N^{\frac{2}{d}}}}|r_{t}(x)|(1-x^{2})^{\frac{d}{2}-1}dx \\
\ll
N^{2}\int\limits_{-1+\frac{c_{1}^{2}}{8N^{\frac{2}{d}}}}^{1-\frac{c_{1}^{2}}{8N^{\frac{2}{d}}}} t^{-\frac{d}{2}-K+s-\frac{1}{2}}(\sqrt{1-x^{2}})^{-\frac{d}{2}-K-\frac{1}{2}}(1-x^{2})^{\frac{d}{2}-1}dx \\
+N^{2}\int\limits_{-1+\frac{c_{1}^{2}}{8N^{\frac{2}{d}}}}^{1-\frac{c_{1}^{2}}{8N^{\frac{2}{d}}}}
t^{-\frac{d}{2}-K+s-\frac{3}{2}}(\sqrt{1-x^{2}})^{-\frac{d}{2}-K-\frac{3}{2}}(1-x^{2})^{\frac{d}{2}-1}dx.
\end{multline}
For the first term in right-hand side of (\ref{ineq5}) we obtain 
\begin{multline}\label{ineq6}
\int\limits_{-1+\frac{c_{1}^{2}}{8N^{\frac{2}{d}}}}^{1-\frac{c_{1}^{2}}{8N^{\frac{2}{d}}}}(\sqrt{1-x^{2}})^{-\frac{d}{2}-K-\frac{1}{2}}(1-x^{2})^{\frac{d}{2}-1}dx=
2\int\limits_{\alpha_{N}}^{\frac{\pi}{2}} (\sin y)^{\frac{d}{2}-K-\frac{3}{2}}dy \\
\ll
\int\limits_{\alpha_{N}}^{\frac{\pi}{2}}  y^{\frac{d}{2}-K-\frac{3}{2}}dy\ll
 (\alpha_{N})^{\frac{d}{2}-K-\frac{1}{2}} \ll 
  \Big(N^{-\frac{1}{d}}) \Big)^{\frac{d}{2}-K-\frac{1}{2}} 
  \ll t^{-\frac{d}{2}+K+\frac{1}{2}},
\end{multline}
where we have used relations (\ref{sinIneq}) and (\ref{NandT}) and fact that $K>\frac{d}{2}+1$.

In the same way
\begin{align}\label{ineq7}
\int\limits_{-1+\frac{c_{1}^{2}}{8N^{\frac{2}{d}}}}^{1-\frac{c_{1}^{2}}{8N^{\frac{2}{d}}}}(\sqrt{1-x^{2}})^{-\frac{d}{2}-K-\frac{3}{2}}(1-x^{2})^{\frac{d}{2}-1}dx\ll t^{-\frac{d}{2}+K+\frac{3}{2}}.
\end{align}

Thus, relations (\ref{ineq5})-(\ref{ineq7}) yield
\begin{align}\label{ineq8}
\Big|N^{2}\int\limits_{\mathbb{S}^{d}\setminus S(\pm \mathbf{y};\alpha_{N})}r_{t}(\langle\mathbf{x},\mathbf{y}\rangle)d\sigma_{d}(\mathbf{x})
\Big| \ll N^{1+\frac{s}{d}}.
\end{align}
Combining (\ref{W_def}), (\ref{ineq2}), (\ref{firstTerm}) and (\ref{ineq8}) we obtain desired estimate (\ref{W_estim}).

Formulas (\ref{main_term}), (\ref{ineq_s_t1}), (\ref{s_eq2}) and (\ref{W_estim}) imply (\ref{Lemma2}).
Lemma~\ref{lem2} is proved.
\end{proof}

\section*{References}

\providecommand{\bysame}{\leavevmode\hbox to3em{\hrulefill}\thinspace}
\providecommand{\MR}{\relax\ifhmode\unskip\space\fi MR }
\providecommand{\MRhref}[2]{%
  \href{http://www.ams.org/mathscinet-getitem?mr=#1}{#2}
}
\providecommand{\href}[2]{#2}


\begin{thebibliography}{10}

\bibitem{Bondarenko-Radchenko-Viazovska2013:optimal_designs}
A.~Bondarenko, D.~Radchenko, and M.~Viazovska, \emph{Optimal asymptotic bounds
  for spherical designs}, Ann. of Math. (2) \textbf{178} (2013), no.~2,
  443--452.

\bibitem{Bondarenko-Radchenko-Viazovska2015:Well_separated}
\bysame, \emph{Well-separated spherical designs}, Constr. Approx. \textbf{41}
  (2015), no.~1, 93--112.

\bibitem{BoyvalenkovDragnevHardinSaffStoyanova}
P.~G. Boyvalenkov, P.~D. Dragnev, D.~P. Hardin, E.~B. Saff, and M.~M.
  Stoyanova, \emph{Universal upper and lower bounds on energy of spherical
  designs}, Dolomites Res. Notes Approx. \textbf{8} (2015), no.~Special Issue,
  51--65.

\bibitem{Brauchart-Hesse2007:numerical_integration}
J.~S. Brauchart and K.~Hesse, \emph{Numerical integration over spheres of
  arbitrary dimension}, Constr. Approx. \textbf{25} (2007), no.~1, 41--71.

\bibitem{Brauchart-Saff-Sloan+2014:qmc_designs}
J.~S. Brauchart, E.~B. Saff, I.~H. Sloan, and R.~S. Womersley, \emph{Q{MC}
  designs: optimal order quasi {M}onte {C}arlo integration schemes on the
  sphere}, Math. Comp. \textbf{83} (2014), no.~290, 2821--2851.

\bibitem{AnChenSloanWomersley_Well2010}
I.H.~Sloan C.~An, X.~Chen and R.S. Womersley, \emph{Well conditioned spherical
  designs for integration and interpolation on the two-sphere}, SIAM J. Numer.
  Anal. \textbf{48} (2010), 2135--2157.

\bibitem{Delsarte-Goethals-Seidel1977:spherical_designs}
P.~Delsarte, J.~M. Goethals, and J.~J. Seidel, \emph{Spherical codes and
  designs}, Geometriae Dedicata \textbf{6} (1977), no.~3, 363--388.

\bibitem{GiganteLeopardi2017:Diameter}
G.~Gigante and P.~Leopardi, \emph{Diameter bounded equal measure partitions of
  ahlfors regular metric measure spaces}, Discrete Comput Geometry \textbf{57}
  (2017), no.~2, 419--430.

\bibitem{Hesse:2009s-energy}
K.~Hesse, \emph{The s-energy of spherical designs on ${S^{2}}$}, Advances in
  Computational Mathematics \textbf{30} (2009), no.~1, 37--59.

\bibitem{HesseLeopardiTheCoulombEnergy}
K.~Hesse and P.~Leopardi, \emph{The {C}oulomb energy of spherical designs on
  ${S^{2}}$}, Advances in Computational Mathematics \textbf{28} (2008), no.~4,
  331--354.

\bibitem{KorevaarMeyers:SphericalFaraday1993}
J.~Korevaar and J.~L.~H. Meyers, \emph{Spherical {F}araday cage for the case of
  equal point charges and {C}hebyshev-type quadrature on the sphere}, Integral
  Transform. Spec. Funct. \textbf{1} (1993), no.~2, 105--117.

\bibitem{KuijlaarsSaff:1998Asymptotics}
A.B.J. Kuijlaars and E.B. Saff, \emph{Asymptotics for minimal discrete energy
  on the sphere}, Trans. Am.Math. Soc. \textbf{350} (1998), no.~2, 523--538.

\bibitem{Magnus-Oberhettinger-Soni1966:formulas_theorems}
W.~Magnus, F.~Oberhettinger, and R.~P. Soni, \emph{Formulas and theorems for
  the special functions of mathematical physics}, Third enlarged edition. Die
  Grundlehren der mathematischen Wissenschaften, Band 52, Springer-Verlag New
  York, Inc., New York, 1966.

\bibitem{Mueller1966:spherical_harmonics}
C.~M\"uller, \emph{Spherical harmonics}, Lecture Notes in Mathematics, vol.~17,
  Springer-Verlag, Berlin-New York, 1966.

\bibitem{GrabnerStepanyuk2018}
Grabner P. and Stepanyuk T., \emph{Upper and lower estimates for numerical
  integration errors on spheres of arbitrary dimension}, arXiv: 1801.05474.

\bibitem{RakhmanovSaffZhou:1994Minimal}
E.~A. Rakhmanov, E.~B. Saff, and Y.~M. Zhou, \emph{Minimal discrete energy on
  the sphere}, Math. Res. Lett. \textbf{1} (1994), no.~6, 547--662.

\bibitem{Schoenberg:1942PositiveDefinite}
I.~J. Schoenberg, \emph{Positive definite functions on spheres}, Duke Math. J.
  \textbf{9} (1942), 96--108.

\bibitem{Szegoe1975:orthogonal_polynomials}
G.~Szeg\H{o}, \emph{Orthogonal polynomials}, fourth ed., American Mathematical
  Society, Providence, R.I., 1975, American Mathematical Society, Colloquium
  Publications, Vol. XXIII.

\bibitem{Wagner:1992UpperBounds}
G.~Wagner, \emph{On means of distances on the surface of a sphere. {II} (upper
  bounds)}, Pacific. J.Math. \textbf{154} (1992), no.~2, 381--396.

\end{thebibliography}


\end{document}